\documentclass{article}
\usepackage{graphicx}
\usepackage{fix-cm}
\usepackage{mathptmx}      
\usepackage{amsmath}
\usepackage{amssymb}
\usepackage{amsfonts}
\usepackage{amsthm}
\usepackage{hyperref}

\newcommand{\mB}{\mathcal B}

\newcommand{\mH}{\mathcal H}
\newcommand{\mJ}{\mathcal J}

\newcommand{\B}{\mathbb B}

\newcommand{\R}{\mathbb R}
\newcommand{\reals}{\mathbb R}
\newcommand{\T}{\mathbb T}

\newcommand{\epi}{\mathop{\rm epi}\nolimits}
\newcommand{\cl}{\mathop{\rm cl}\nolimits}

\newcommand{\dom}{\mathop{\rm dom}\nolimits}
\newcommand{\li}{\mathop{\rm liminf}}
\newcommand{\ls}{\mathop{\rm limsup}}

\newcommand{\mli}{\mathop{\rm \mu\text{-}liminf}}
\newcommand{\mls}{\mathop{\rm \mu\text{-}limsup}}

\newcommand{\interior}{\mathop{\rm int}}
\newcommand{\rinterior}{\mathop{\rm rint}}
\newcommand{\esssup}{\mathop{\rm esssup}}
\newcommand{\supp}{\mathop{\rm supp}}

\def\keywords{\vspace{.5em}
{\textit{Keywords}:\,\relax%
}}

\newtheorem{theorem}{Theorem}
\newtheorem{lemma}{Lemma}
\newtheorem{corollary}{Corollary}
\newtheorem{proposition}{Proposition}
\newtheorem{example}{Example}
\newtheorem{remark}{Remark}

\begin{document}
\title{Continuous essential selections and integral functionals} 

\author{Ari-Pekka Perkki\"o \footnote{Department of Mathematics and Systems Analysis, Aalto University, P.O. Box 11100, FI-00076 Aalto, Finland, \href{mailto:ari-pekka.perkkio@aalto.fi}{ari-pekka.perkkio@aalto.fi}}}


\maketitle
\begin{abstract}
Given a strictly positive measure, we characterize inner semicontinuous solid convex-valued mappings for which continuous functions which are selections almost everywhere are selections. This class contains continuous mappings as well as fully lower semicontinuous closed convex-valued mappings that arise in variational analysis and optimization of integral functionals. The characterization allows for extending existing results on convex conjugates of integral functionals on continuous functions. We also give an application to integral functionals on left continuous functions of bounded variation.
\keywords{Set-valued and variational analysis \and Continuous selections \and Integral functionals \and Convex duality}
\end{abstract}
\section{Introduction}
Given a set-valued mapping $\Gamma$ from a topological space $\T$ to another $X$ and a strictly positive\footnote{$\mu$ is strictly positive if $\mu(O)>0$ for every nonempty open $O$.} countably additive Borel measure $\mu$ on $\T$, we say that a function $y:\T\rightarrow X$ is an {\em essential selection} of $\Gamma$ if $y_t\in\Gamma_t$ $\mu$-almost everywhere. In this article we study in which situation continuous essential selections are selections. This is important when deriving formulae for convex conjugates of integral functionals on continuous functions. Such conjugates arise in several areas of variational analysis and optimization including optimal control, plasticity theory and mathematical finance, see \cite{afp00,abm5,bv88,ks9,roc76b} and references therein. 

A function, which belongs to the domain of an integral functional, is an essential selection of the domain of the integrand. However, without further conditions such function is not necessarily a selection of the domain. The articles \cite{bv88,roc71} study convex conjugates of integral functionals on continuous functions. In \cite[Theorem 5]{roc71} Rockafellar assumes full lower semicontinuity, which implies that continuous essential selections are selections of the closure of the domain, whereas in \cite[Example 5]{bv88} this property was assumed explicitly.

We show that, for an inner semicontinuous solid convex-valued mapping, continuous essential selections are selections if and only if the mapping has a property which we call {\em outer regularity in measure}. While full lower semicontinuity is a purely topological notion, outer regularity in measure takes into account the underlying measure. When $\T\subset \R^n$, we prove that fully lower semicontinuous closed convex-valued mappings are outer regular in measure which allows for a generalization \cite[Theorem 5]{roc71} about convex conjugates of integral functionals on continuous functions.  We also give an application to integral functionals on left continuous functions of bounded variation. These results have further applications to problems of Bolza; see \cite{pp13}.

\section{Outer regularity in measure}\label{sec:outerregularity}
We denote by $\mH_t$ and $\mH_x$ the neighborhood systems of $t\in\T$ and $x\in X$, respectively. We will denote by $\mH^o_x$ the system of open neighborhoods of $x\in X$. The interior and closure of a set $A$ will be denoted by $\interior A$ and $\cl A$. 

Let $\Gamma:\T\rightrightarrows X$ be a set-valued mapping. The {\em outer limit} and {\em inner limit} are, respectively,
\begin{align}
\begin{split}\label{eq:lsli}
	(\ls \Gamma)_t &=\bigcap_{B\in\mH_t}\cl\left(\bigcup_{t'\in B} \Gamma_{t'}\right),\\
	(\li \Gamma)_t &=\bigcap_{B\in\mH_t^\#}\cl\left(\bigcup_{t'\in B} \Gamma_{t'}\right),
\end{split}
\end{align}
where $\mH_t^\#=\{B\subset \T\mid  B\cap O\neq\emptyset\ \forall\, O\in\mH_t\}$ is the {\em grill} of $\mH_t$; see \cite{bw96}. The mapping $\Gamma$ is {\em outer semicontinuous} or {\em inner semicontinuous}, respectively, if 
\begin{align*}
	(\ls \Gamma)_t \subset \Gamma_t \quad&\forall t,\\
	\Gamma_t \subset (\li \Gamma)_t \quad&\forall t.
\end{align*}
We refer to \cite[Chapter 5]{rw98} for a systematic treatment of these concepts in the finite dimensional case. The above limits can also be expressed as
\begin{equation}\label{eq:liminf}
\begin{split}
	(\ls\Gamma)_t &=\left\{x\in X\mid \Gamma^{-1}(A)\in\mH_t^{\#}\ \forall A\in\mH_x^o\right\},\\
	(\li\Gamma)_t &=\left\{x\in X\mid \Gamma^{-1}(A)\in\mH_t\ \forall A\in\mH_x^o\right\},
\end{split}
\end{equation}
where $\Gamma^{-1}(A)=\{t\mid \Gamma_t\cap A\neq\emptyset\}$. In \cite{bw96} this was stated for a metric space $X$; for general $X$, \eqref{eq:liminf} can be verified analogously to the proof of Lemma~\ref{lem:mlim1} below. Moreover, $\Gamma$ is inner semicontinuous if and only if $\Gamma^{-1}(A)$ is an open set for every open $A\subset X$ \cite[Proposition 2.1]{mic56}. This is taken as the definition of lower semicontinuity in \cite{mic56}.

Given a strictly positive countably additive measure $\mu$ on the Borel $\sigma$-algebra $\mB(\T)$, we define $\mH_t^\mu=\{B\in\mathcal B(\T)\mid \exists O\in\mH_t: \mu(B\cap O)=\mu(O)\}$ and
\begin{align*}
	(\mli \Gamma)_t&=\bigcap_{B\in\mH_t^{\mu\#}}\cl\left(\bigcup_{t'\in B} \Gamma_{t'}\right),
\end{align*}
where $\mH_t^{\mu\#}=\{B\in\mB(\T)\mid B\cap O\neq\emptyset\ \forall O\in \mH^\mu_t\}$. It is easily verified that
\[
	\mH_t^{\mu\#}=\left\{B\in\mB(\T)\mid \mu(B\cap O)>0\ \forall O\in \mH_t\right\}.
\]
We have $\mH_t\subset \mH_t^{\mu\#}\subset \mH_t^\#$ and 
\begin{equation}
(\li \Gamma)_t\subseteq (\mli \Gamma)_t \subseteq (\ls \Gamma)_t.\label{eq:inc}
\end{equation}
All these limits are closed and they are invariant under the image closure. Here the {\em image closure} of $\Gamma$ is defined as $\cl\Gamma_t$ for every $t$. Moreover, $\mli$ is invariant under equivalent changes of measure.
\begin{example}
Assume that $X$ is Hausdorff and that $w:\T\rightarrow X$ has a {\em continuous modification}, i.e., there is a continuous $y:\T\rightarrow X$ such that $y_t=w_t$ $\mu$-almost everywhere. Since $y$ is continuous, we have $\li y_t=\ls y_t$ so that by \eqref{eq:inc}, $\mli y_t= y_t$ (we identify $y$ with $t\mapsto \{y_t\}$). By Lemma~\ref{lem:mlim0} below, we get $\mli w_t= y_t$.
\end{example}

We say that $\Gamma$ is {\em outer regular in measure} or {\em outer $\mu$-regular} if 
\begin{align*}
	(\mli \Gamma)_t &\subseteq \cl\Gamma_t\quad\forall t.
\end{align*} 
Outer regularity in measure is invariant under the image closure and under equivalent changes of measure.  The sets of continuous selections and continuous essential selections of $\Gamma$ will be denoted by $C(\T;\Gamma)$ and $C(\T,\mu;\Gamma)$,  respectively.
\begin{theorem}\label{thm:mlim0}
If $\Gamma$ is closed-valued and outer $\mu$-regular, then 
\[
C(\T,\mu;\Gamma) = C(\T;\Gamma).
\]
\end{theorem}
\begin{proof}
Let $y\in C(\T,\mu;\Gamma)$, and let $N\in\mB(\T)$ be a $\mu$-null set such that $\{t'\mid y_{t'}\notin\Gamma_{t'}\}\subset N$. Let $B\in\mH_t^{\mu\#}$. Since $(B\backslash N)\cap O_t\neq\emptyset$ for every $O_t\in\mH_t$, there is a sequence $(t^\nu)_{\nu=1}^\infty$ in $B\backslash N$ such that $t^\nu\rightarrow t$. Since $y_{t^\nu}\in\Gamma_{t^\nu}$ for every $\nu$, and since $y$ is continuous, we get $y_t\in \cl(\bigcup_{t'\in B} \Gamma_{t'})$. Because this holds for every $B\in\mH_t^{\mu\#}$, we have $y_t\in (\mli \Gamma)_t$. Since $\Gamma$ is outer $\mu$-regular, we get $y\in C(\T;\Gamma)$. The opposite inclusion $C(\T;\Gamma)\subseteq C(\T,\mu;\Gamma)$ is trivial.
\end{proof}
\begin{remark}
Theorem~\ref{thm:mlim0} and the other results below hold for non-closed mappings with appropriate reformulations using image closures.
\end{remark}
Section~\ref{sec:properties} will be concerned with properties of $\mli$ which are applied in Section~\ref{sec:necessity} to prove a converse of Theorem~\ref{thm:mlim0}. Our aim is to show that outer $\mu$-regularity is also a necessary condition in Theorem~\ref{thm:mlim0}.
\begin{remark}
Analogously to \eqref{eq:lsli}, one can define
\begin{align*}
	(\mls\Gamma)_t&=\bigcap_{B\in\mH_t^{\mu}}\cl\left(\bigcup_{t'\in B} \Gamma_{t'}\right).
\end{align*}
However, we do not analyze this limit concept here. 
\end{remark}

\subsection{Properties of $\mli$}\label{sec:properties}

The following result is comparable with \eqref{eq:liminf}. Recall that $\Gamma$ is {\em measurable} if $\Gamma^{-1}(A)\in\mathcal B(\T)$ for every open $A\subseteq X$. 
\begin{lemma}\label{lem:mlim1}
If $\Gamma$ is measurable, then
\begin{align*}
	(\mli\Gamma)_t &=\left\{x\in X\mid \Gamma^{-1}(A)\in\mH_t^\mu\ \forall A\in\mH_x^o\right\}.
\end{align*}
\end{lemma}
\begin{proof}
Assuming that there is an $A\in\mH_x^o$ such that $\mu(O_t\cap\Gamma^{-1}(A))<\mu(O_t)$ for all $O_t\in\mH_t$, we have $(\Gamma^{-1}(A))^C\in\mH_t^{\mu\#}$, which implies $x\notin (\mli \Gamma)_t$, because $x\notin\cl (\bigcup_{t\in (\Gamma^{-1}(A))^C}\Gamma_t)$. 

Assuming that $\Gamma^{-1}(A)\in\mH_t^\mu$ for all $A\in\mH_x^o$, we have $\cl(\bigcup_{t\in B}\Gamma_t)\cap A\neq\emptyset$ for all \mbox{$B\in\mH_t^{\mu\#}$} and $A\in\mH_x^o$, which implies 
\begin{align*}
 	\{x\}\subseteq\bigcap_{B\in\mH_t^{\mu\#}}\left[\bigcap_{A\in\mH_x^o}\cl\left(\bigcup_{t\in B}\Gamma_t\right)\cap A\right] \subset(\mli\Gamma)_t.
\end{align*}
\end{proof}
The following lemma shows that $\mli\Gamma$ is invariant under changes of $\Gamma$ on $\mu$-null sets.
\begin{lemma}\label{lem:mlim0}
If $\Gamma:\T\rightrightarrows X$ and $\tilde \Gamma:\T\rightrightarrows X$ are closed-valued and $\Gamma_t=\tilde \Gamma_t$ $\mu$-a.e., then 
\[
\mli\Gamma_t=\mli\tilde\Gamma_t\quad\forall t.
\]
\end{lemma}
\begin{proof}
There is a Borel $\mu$-null set $N$ such that $\Gamma_t=\tilde \Gamma_t$ for every $t\in N^C$. Since, for all \mbox{$B\in\mB(\T)$,} we have that $B\cap N^C\in\mH_t^{\mu\#}$ if and only if $B\in\mathcal H_t^{\mu\#}$, we get
\begin{align*}
	\bigcap_{B\in\mH_t^{\mu\#}}\cl\left(\bigcup_{t'\in B} \Gamma_{t'}\right) &\subseteq\bigcap_{B\in\mH_t^{\mu\#}}\left\{\cl\left(\bigcup_{t'\in B\cap N^C} \Gamma_{t'}\right)\mid B\cap N=\emptyset\right\}\\
	&=\bigcap_{B\in\mH_t^{\mu\#}}\cl\left(\bigcup_{t'\in B\cap N^C} \Gamma_{t'}\right)\\
	&\subseteq \bigcap_{B\in\mH_t^{\mu\#}}\cl\left(\bigcup_{t'\in B} \Gamma_{t'}\right),
\end{align*}
which gives 
\[
\mli \Gamma_t= \bigcap_{B\in\mH_t^{\mu\#}}\cl\left(\bigcup_{t'\in B\cap N^C} \Gamma_{t'}\right).
\]
This implies that $\mli \Gamma_t= \mli \tilde \Gamma_t$ for all $t$.
\end{proof}

Recall that $\T$ is {\em Lindel\"of} if every open cover of $\T$ has a countable subcover, and that $\T$ is {\em strongly Lindel\"of} if every subspace of $\T$ is Lindel\"of. When $X$ is a normed space, $\B(x,r)$ denotes the open ball with center $x$ and radius $r$, and $d(A,B)=\inf\{\|x-x'\|\mid x\in A,\, x'\in B\}$ denotes the distance between two sets $A$ and $B$. A convex-valued $\Gamma$ is {\em solid} if it is closed-valued and $\interior\Gamma_t\neq\emptyset$ for all $t$.

\begin{proposition}\label{prop:mli}
Assume that $\T$ is strongly Lindel\"of, $X=\reals^d$ and that $\Gamma$ is measurable and closed-valued.
\begin{enumerate}
\item[{\rm(A)}] If $\Gamma$ is convex-valued, then $t\mapsto\mli\Gamma_t$ is convex-valued.
\item[{\rm(B)}] If $\Gamma$ is inner semicontinuous solid convex-valued, then $\mli\Gamma$ is inner semicontinuous solid convex-valued, $\mli\Gamma_t=\mli\Gamma_t$ $\mu$-a.e. and 
\[
\mli(\mli\Gamma)=\mli\Gamma.
\]
\end{enumerate}
\end{proposition}
\begin{proof}
Let $(A^\nu)_{\nu=1}^\infty$ be a countable open base for the topology on $X$. For every $\nu$ and $t$ with $t\in(\mli\Gamma)^{-1}(A^\nu)$, there is, by Lemma~\ref{lem:mlim1}, an open $O^\nu_t\in\mH_t$ and a $\mu$-null set $N_t^\nu\in\mB(\T)$ such that $O^\nu_t\backslash\Gamma^{-1}(A^\nu)=N^\nu_t$. Since $\T$ is strongly Lindel\"of, there is a countable $\mJ^\nu$ (here $\mJ^\nu=\emptyset$ if $(\mli\Gamma)^{-1}(A^\nu)=\emptyset$) for which $(\mli\Gamma)^{-1}(A^\nu)\subset\bigcup_{t\in\mJ^\nu}O^\nu_t$. Let 
\[
N=\bigcup_{\nu,t\in\mathcal J^\nu}N^\nu_t,
\]
which is a $\mu$-null-set.

Let $t'$ be such that $\mli\Gamma_{t'}\not\subseteq\Gamma_{t'}$. There is an $A^\nu$ such that $\mli\Gamma_{t'}\cap A^\nu\neq\emptyset$ but $\Gamma_{t'}\cap A^\nu=\emptyset$. Choose $t\in\mathcal J^\nu$ such that $t'\in O^\nu_t$. Since $O^\nu_t\backslash\Gamma^{-1}(A^\nu)=N^\nu_t$ and since $t'\notin\Gamma^{-1}(A^\nu)$, we have $t'\in N$. Thus $\mli\Gamma_t\subseteq\Gamma_t$ for all $t\in N^C$.

Assume now that $\Gamma$ is convex-valued. Let $x^1,x^2\in(\mli\Gamma)_{t}$ and $\alpha\in[0,1]$, and denote $\bar x=\alpha x^1+(1-\alpha)x^2$. Let $A\in\mH_{\bar x}^o$ be convex. By Lemma~\ref{lem:mlim1} and the construction of $N$, there is $O^1_t,O^2_t\in\mH_t$ such that $\Gamma^{-1}(A-\bar x + x^1)=O^1_t\backslash N$ and $\Gamma^{-1}(A-\bar x + x^2)=O^2_t\backslash N$. For any $t'\in O^2_t\cap O^1_t\backslash N$, there is an $x^1_{t'}\in\Gamma_{t'}\cap (A-\bar x+x^1)$ and an $x^2_{t'}\in\Gamma_{t'}\cap (A-\bar x+x^2)$ so that $\alpha x^1_{t'}+(1-\alpha)x^2_{t'}\in\Gamma_{t'}\cap A$, which implies $\Gamma^{-1}(A)\in\mH^\mu_{t}$. Since this holds for every convex $A\in\mH_{\bar x}^o$, and since every element of $\mH_{\bar x}^o$ contains some convex $A\in\mH_{\bar x}^o$, we get, by Lemma~\ref{lem:mlim1}, that $\bar x\in\mli\Gamma_t$.

Assume now that $\Gamma$ is inner semicontinuous solid convex-valued. By inner semicontinuity of $\Gamma$, \eqref{eq:inc} and by $\mli\Gamma_t\subseteq\Gamma_t$ for all $t\in N^C$ , we get 
\begin{align}\label{eq2:N}
\mli\Gamma_t=\Gamma_t\quad \forall t\in N^C.
\end{align}
Thus, by Lemma~\ref{lem:mlim0}, we have $\mli(\mli\Gamma)=\mli\Gamma$.

By \eqref{eq:inc}, $\interior\mli\Gamma_t\neq\emptyset$ for all $t$. Since $\mli\Gamma$ is convex-valued, to prove inner semicontinuity of $\mli\Gamma$, it suffices to show that $\bar x\in \li\mli\Gamma_{\bar t}$ whenever $\bar x\in\interior\mli\Gamma_{\bar t}$. Let $(x^i)_{i=1}^d$ be an orthonormal basis in $\R^d$ and $\epsilon>0$ be such that $\bar x\pm\epsilon x^i\in\mli\Gamma_{\bar t}$ for all $i$. Since there are finitely many $x^i$, there is, by Lemma~\ref{lem:mlim1}, an open $O_{\bar t}\in\mH_{\bar t}$ such that $\Gamma_t\cap\B(\bar x\pm\epsilon x^i,\epsilon/2)\neq\emptyset$ $\mu$-a.e. on $O_{\bar t}$ for all $x^i$. Since $\Gamma$ is convex-valued, we have $\bar x\in\Gamma_t$ $\mu$-a.e. on $O_{\bar t}$. Therefore, by Lemma~\ref{lem:mlim1}, $\bar x\in\mli\Gamma_t$ for all $t\in O_{\bar t}$ and consequently $\bar x\in\li\mli\Gamma_{\bar t}$.
\end{proof}
The following example illustrates that, without convexity, inner semicontinuity is not necessarily preserved under $\mli$.
\begin{example}\label{ex:nisc}
Let $\T=[0,1]$ be equipped with the standard topology, $\mu$ be the Lebesque measure and let $\tilde\Gamma_t= \{2^{-n}\}$ for $t\in (2^{-(n+1)},2^{-n})$ and $n\in\mathbb N$. Let $\Gamma_t=\tilde\Gamma_t\cup[1,2]$ so that $\Gamma$ is inner semicontinuous closed nonempty-valued. It follows from Lemma~\ref{lem:mlim1} that $\mli\Gamma_0=\{0\}\cup[1,2]$ but $\mli \Gamma_{t}=[1,2]$ whenever $t=2^{-n}$ for some $n\in\mathbb N$. In particular, $t\mapsto\mli \Gamma_t$ is not inner semicontinuous at the origin.
\end{example}

A closed convex $\R^d$-valued $\Gamma$ is {\em fully lower semicontinuous} if it is inner semicontinuous solid-valued and $x\in\cl\Gamma_t$ whenever there exist $A\in\mH_x$ and $O\in\mH_t$ such that $\{t\in O\mid A\subset \Gamma_t\}$ is dense in~$O$. In \cite[Theorem 5]{roc71} Rockafellar uses full lower semicontinuity to prove an explicit expression for the convex conjugate of an integral functional on continuous functions.
\begin{lemma}\label{lem:fisc}
Assume that $\T\subseteq\R^n$. If $\Gamma$ is closed convex $\R^d$-valued and fully lower semicontinuous, then $\Gamma$ is outer regular in measure. 
\end{lemma}
\begin{proof}
By Proposition~\ref{prop:mli}, the mapping $\mli\Gamma$ is inner semicontinuous solid convex-valued. It suffices to show that, for any $t\in\T$ and $x\in\interior\mli \Gamma_t$, we have $x\in\cl\Gamma_t$. 

Let $\epsilon>0$ be such that $\cl \B(x,\epsilon)\subset \mli\Gamma_t$. By \cite[Proposition 1.6]{bw96}, there is $O_t\in\mH_t$ such that $\cl\B(x,\epsilon)\subset \mli\Gamma_t+\cl\B(0,\epsilon/2)$ for all $t\in O_t$. Thus $\B(x,\epsilon/2)\subset\mli\Gamma_t$ for all $t\in O_t$. By Proposition~\ref{prop:mli}, $\mli\Gamma_t=\Gamma_t$ $\mu$-a.e. so that there is a $\mu$-null set $N\in\mB(\T)$ such that $\B(x,\epsilon/2)\subset\Gamma_t$ for all $t\in O_t\cap N^C$. Since $\Gamma$ is fully lower semicontinuous and $\{t\in O_t\mid \B(x,\epsilon/2)\subset\Gamma_t\}$ is dense in $O_t$, we get $x\in\cl\Gamma_t$. 
\end{proof}
The following example demonstrates that there are inner semicontinuous solid convex-valued mappings which are outer $\mu$-regular but not fully lower semicontinuous.
\begin{example}
Let $\T\subseteq \R^n$ be open and $S\subset\T$ be a $(n-1)$-dimensional closed set such that $\T\backslash S$ is dense in $\T$. Let $H$ be the restriction of $(n-1)$-dimensional Hausdorff measure to $S$, and assume that $\supp(H)=S$ (see, e.g., \cite{afp00}). Let $\mu=\lambda+H$, where $\lambda$ is the Lebesque measure on $\T$. Define 
\[
	\Gamma_t= \begin{cases}
			\cl{\mathbb B}(0,r_2)\quad&\text{if } t\notin S,\\
			\cl{\mathbb B}(0,r_1)\quad&\text{if } t\in S,
			\end{cases}
\]
where $0\le r_1<r_2\le+\infty$. Since $S$ is a closed set, $\Gamma^{-1}(A)$ is open for every open $A$, and therefore $\Gamma$ is inner semicontinuous. To check outer $\mu$-regularity of $\Gamma$, it suffices to consider $t\in S$ and $x\notin \cl \mathbb B(0,r_1)$. Let $A\in\mathcal H_x^o$ be such that $\cl\mathbb B(0,r_1)\cap A=\emptyset$. Since $\supp (H)=S$, we have $\mu(\Gamma^{-1}(A)\cap O)\le\lambda(O)<\mu(O)$ for every $O\in\mathcal H_t$ so that, by Lemma~\ref{lem:mlim1}, $x\notin\mli\Gamma_t$. Thus $\mli\Gamma_t\subseteq\cl\Gamma_t$ for all $t$ and $\Gamma$ is outer $\mu$-regular.

Assume now that $0<r_1$. Let $x\in\B(0,r_2)$. There is an $A\subset\mH_x$ such that $A\subset\cl\B(0,r_2)$. Thus $\{t\in \T\mid A\subset \Gamma_t\}$ is dense in $\T$, because it contains $\T\backslash S$. However, $x\notin \cl\Gamma_t$ whenever $x\notin \cl\B(0,r_1)$ and $t\in S$. Therefore, $\Gamma$ is not fully lower semicontinuous though it is inner semicontinuous solid convex-valued.
\end{example}

\subsection{Necessity of outer regularity in measure}\label{sec:necessity}
In this section we show that, for inner semicontinuous solid convex $\R^d$-valued mappings, the condition $C(\T,\mu;\Gamma)=C(\T;\Gamma)$ in Theorem~\ref{thm:mlim0} is necessary for outer $\mu$-regularity. The proof is based on \cite[Lemma 5.2]{mic56}, which says that if $\T$ is a perfectly normal $T_1$-space, $X$ is a separable Banach space and $\Gamma$ is an inner semicontinuous closed convex nonempty-valued mapping, then there exists a sequence $(y^\nu)_{\nu=1}^\infty\subset C(\T;\Gamma)$ such that $(y_t^\nu)_\nu^\infty$ is dense in $\Gamma_t$ for every $t$. Such a sequence is usually referred to as a {\em Michael representation} of $\Gamma$.

Recall that a topological space is $T_1$ if for every distinct points $t$ and $t'$ there is an $O_{t}\in\mH_{t}$ with $t'\notin O_{t}$. The space is {\em normal} if for every disjoint closed sets $B$ and $B'$ there are disjoint open sets $O$ and $O'$ such that $B\subset O$ and $B'\subset O'$. The space $\T$ is {\em perfectly normal} if it is normal and every closed set is a countable intersection of open sets.

\begin{theorem}\label{thm:ces}
Assume that $\T$ is a Lindel\"of perfectly normal $T_1$-space. An inner semicontinuous solid convex $\R^d$-valued mapping $\Gamma$ is outer $\mu$-regular if and only if $C(\T,\mu;\Gamma)=C(\T;\Gamma)$.
\end{theorem}
\begin{proof}
Assuming that $\Gamma$ is outer $\mu$-regular, Theorem~\ref{thm:mlim0} implies that $C(\T,\mu;\Gamma)=C(\T;\Gamma)$. Assume that $C(\T,\mu;\Gamma)=C(\T;\Gamma)$. We use the fact that a Lindel\"of perfectly normal space is strongly Lindel\"of (see \cite[p. 194]{eng89}). By Proposition~\ref{prop:mli}, $t\mapsto\mli\Gamma_t$ is an inner semicontinuous closed convex nonempty-valued mapping. By \cite[Lemma 5.2]{mic56}, there is a $(y^\nu)_{\nu=1}^\infty\subseteq C(\T;\mli\Gamma)$ such that $(y_t^\nu)_{\nu= 1}^\infty$ is dense in $\mli\Gamma_t$ for every $t$. Thus 
\[
	\mli\Gamma_t=\cl\{y^\nu_t\mid \nu\ge 1\}\subseteq\Gamma_t\quad\forall t,
\]
where the inclusion follows from Proposition~\ref{prop:mli}, because now $\mli\Gamma_t=\Gamma_t$ $\mu$-a.e. implies $C(\T;\mli\Gamma)=C(\T,\mu;\Gamma)=C(\T;\Gamma)$. This shows that $\Gamma$ is outer $\mu$-regular. 
\end{proof}
\begin{remark}\label{rem:ces}
Let $\bar\Gamma$ be the image closure of $\Gamma$. The proof of Theorem~\ref{thm:ces} actually shows that if $t\mapsto\mli\Gamma$ is inner semicontinuous convex nonempty-valued, then $\Gamma$ is outer $\mu$-regular if and only if $C(\T,\mu;\bar\Gamma)=C(\T;\bar\Gamma)$. In this case the last inclusion in the proof follows from $\mli\Gamma_t\subseteq \cl\Gamma_t$ $\mu$-a.e. (see the proof of Proposition~\ref{prop:mli}).
\end{remark}

Without convexity or inner semicontinuity the necessity in Theorem~\ref{thm:ces} does not hold in general. Indeed, in either situation continuous selections need not exist at all.

\section{Applications to conjugates of integral functionals}\label{sec:5}
From now on we will assume that $\T$ is a Lindel\"of perfectly normal $T_1$-space and $X=\R^d$. Let $h$ be a {\em convex normal integrand} on $\T\times\R^d$, i.e., $h$ is an extended real-valued function and $t\mapsto \epi h_t=\{(x,\alpha)\mid h_t(x)\le \alpha\}$ is closed convex-valued and measurable from $\T$ to $\R^d\times\R$. By \cite[Proposition 14.28]{rw98}, $t\mapsto h_t(w_t)$ is measurable whenever $w:\T\rightarrow\R^d$ is measurable so that the {\em integral functional}
\[
	I_h(y)=\int_\T h_t(y_t)d\mu_t
\]
is well-defined on $C=C(\T;\R^d)$. Here and in what follows the integral of a measurable function is defined as $+\infty$ unless the positive part is integrable. 

Let $C_b=C_b(\T;\R^d)$ be the space of bounded continuous functions and $M_b$ be the space of $\reals^d$-valued finite Radon measures\footnote{A measure $\theta$ is {\em Radon} if it is countably additive and $|\theta|(B)=\sup\{|\theta|(K)\mid K\subseteq B,\ K\text{ is compact}\}$ for every $B\in\mB(\T)$.} on $\T$. The bilinear form 
\[
\langle y,\theta\rangle = \int_\T y_t d\theta_t
\]
 puts the spaces $C_b$ and $M_b$ in separating duality. Indeed, it follows from \cite[p.71]{bog7} that $C_b$ separates the points in $M_b$, whereas it is evident that $M_b$ separates the points of $C_b$.

Our aim is to study the {\em conjugate} 
\[
	I_h^*(\theta)=\sup\{\langle y,\theta\rangle -I_h(y)\}.
\]
In \cite[Theorem 5]{roc71} Rockafellar gave conditions under which the conjugate of $I_h$ can be expressed in terms of the conjugate of $h$ as
\begin{equation*}
	J_{h^*}(\theta)=\int_\T h_t^*\left(\left(\frac{d\theta}{d\mu}\right)_t\right)d\mu_t+\int_{\T} (h^*)^\infty_t\left(\left(\frac{d\theta^s}{d|\theta^s|}\right)_t\right)d|\theta^s|_t,
\end{equation*}
where $\theta^s$ is the singular part of $\theta\in M_b$ with respect to $\mu$. Here the {\em conjugate} of $h$ is defined by $h^*_t(v)=\sup_{x\in\R^d}\{x\cdot v-h_t(x)\}$ and $(h_t^*)^\infty$ denotes the {\em recession function} of $h^*_t$. That is, $(h_t^*)^\infty$ is defined by 
\[
(h^*_t)^\infty(v) = \sup_{\alpha>0}\frac{h^*_t(\alpha v+\bar v)-h^*_t(\bar v)}{\alpha},
\]
where $\bar v\in\dom h^*_t=\{v\in\R^d\mid h^*_t(v)<\infty\}$; see \cite[Chapter 8]{roc70a}. We will use the techniques from \cite{bv88} to generalize Rockafellar's result and to prove a result for  integral functionals on functions of bounded variation. For other related results, see, e.g., \cite{ab88,afp00,abm5,bv88,roc71,val86} and references therein. 

We denote the relative interior of a set $A\subset\R^d$ by $\rinterior A$. Recall that $\mu$ is {\em $\sigma$-finite} if $\T$ is a countable union of sets with finite $\mu$-measure. 
\begin{theorem}\label{thm:if1}
Assume that $\mu$ is a $\sigma$-finite Radon measure, $\dom h$ is inner semicontinuous, $\dom J_{h^*}\neq\emptyset$, $C_b(\T;\rinterior \dom h)\cap\dom I_h\neq\emptyset$ and that for every $y\in C(\T;\rinterior\mli\dom h)$ and for every $t$ there exists $O_t\in\mH_t$ such that 
\[
\int_{O_t}|h_t(y_t)|d\mu_t<\infty.
\]
If $\dom h$ is outer $\mu$-regular, then $I_h$ and $J_{h^*}$ are conjugates of each other. If $\interior\dom h_t\neq\emptyset$ for all $t$ and if $I_h$ and $J_{h^*}$ are conjugates of each other, then $\dom h$ is outer $\mu$-regular.
\end{theorem}
\begin{proof}
We have that
\begin{align}
	J_{h^*}^*(y)&=\sup_{\theta\in M_b}\left\{\int_\T y_td\theta_t-J_{h^*}(\theta)\right\}\nonumber \\
	&=\sup_{\theta'\in L^1(\T,\mu;\R^d)}\left\{\int_\T  y_t\cdot \theta'_td\mu_t-I_{h^*}(\theta')\right\}\nonumber \\
	&\quad+\sup_{\theta\in M_b,\theta^{s'}\in L^1(\T,|\theta^s|;\R^d)}\left\{\int_\T y_t\cdot\theta^{s'}_td|\theta^s|_t-\int_\T (h^*_t)^\infty(\theta^{s'}_t)d|\theta^s|_t\right\}\nonumber \\
	&=\begin{cases}
 	I_h(y)\quad&\text{if } y_t\in\cl\dom h_t\ \forall t,\\
	+\infty\quad&\text{otherwise}.
	\end{cases}\label{eq:if2}
\end{align}
Above the second equality follows from the positive homogeneity of  $(h^*_t)^\infty$ \cite[Theorem 8.5]{roc70a}, and the third equality follows by first applying \cite[Theorem 14.60]{rw98} on the second and the third line, where one uses the fact that the indicator function $\delta_{\cl\dom h_t}$ is the conjugate of $(h^*_t)^\infty$ \cite[Theorem 13.3]{roc70a}; and then taking supremum over all purely atomic finite measures which are singular with respect to $\mu$. On the other hand, since $\dom J_{h^*}\neq\emptyset$, there is a $\bar \theta\in \dom J_{h^*}$ which is absolutely continuous with respect to $\mu$. Therefore, for all $w\in L^\infty(\T,\mu;\R^d)$,
\begin{equation}\label{eq:if3}
	\int_\T h_t(w_t)d\mu_t\ge \int_\T \left[w_t\cdot\left(\frac{d\bar\theta}{d\mu}\right)_t-h_t^*\left(\left(\frac{d\bar\theta}{d\mu}\right)_t\right)\right]d\mu_t >-\infty.
\end{equation}

Assume first that $\dom h$ is outer $\mu$-regular. By Theorem~\ref{thm:mlim0}, $y_t\in\cl\dom h_t$ for every $t$ whenever $y\in\dom I_h$ so that \eqref{eq:if2} implies that $J_{h^*}^*=I_h$. Therefore, $I_h^*\le J_{h^*}$ and consequently it suffices to show that $I_h^*\ge J_{h^*}$.

Since $|\theta^s|$ and $\mu$ are singular, there is a $\bar B\in\mB(T)$ such that $\mu(\bar B^C)=|\theta^s|(\bar B)=0$. We define 
\begin{align*}
	\bar h_t(x)=
	\begin{cases} 
		-x\cdot\left(\frac{d\theta}{d\mu}\right)_t+h_t(x)\quad&\text{if } t\in \bar B\\
		-x\cdot\left(\frac{d\theta^s}{d|\theta^s|}\right)_t+\delta_{\cl\dom h_t}(x)\quad&\text{if } t\in \bar B^C,
	\end{cases}
\end{align*}
which is a normal integrand (see \cite[Chapter 14]{rw98}) with $\cl\dom \bar h_t=\cl\dom h_t$ for all $t$. Since $y_t\in \cl\dom h_t$ for all $t$ whenever $I_h(y)<\infty$ and since, by \eqref{eq:if3}, $I_h(y)>-\infty$ for all $y$, we have
\begin{align*}
I_h^*(\theta) &= -\inf_{y\in C_b}\Big\{\int_{\bar B} \left[-y_t\cdot\left(\frac{d\theta}{d\mu}\right)_t+h_t(y_t)\right]d\mu_t\\
&\qquad\qquad+\int_{\bar B^C}\left[-y_t\cdot \left(\frac{d\theta^s}{d|\theta^s|}\right)_t+\delta_{\cl\dom h_t}(y_t)\right]d|\theta^s|_t\Big\}\\
&= - \inf_{y\in C_b}I_{\bar h}^\theta(y),
\end{align*}
where we defined $\mu^\theta=\mu+|\theta^s|$ and $I_{\bar h}^\theta(w)=\int_\T \bar h_t(w_t)d\mu^\theta_t$ for any measurable $w:\T\rightarrow\R^d$. By \cite[Theorem 14.60]{rw98},
\begin{align}
	\inf_{w\in L^\infty(\T,\mu^\theta;\R^d)} I_{\bar h}^\theta(w)&= \int_\T \inf_{x\in\R^d}\bar h_t(x)d\mu^\theta_t\nonumber \\
 &=-\int_\T h_t^*\left(\left(\frac{d\theta}{d\mu}\right)_t\right)d\mu_t - \int_\T (h_t^*)^\infty\left(\left(\frac{d\theta^s}{d|\theta^s|}\right)_t\right)d|\theta^s|_t \nonumber\\
 &=-J_{h^*}(\theta).\label{eq:barh}
\end{align}

Let us show that we can restrict the infimum in the above expression to $w\in L^\infty(\T,\mu^\theta;\R^d)$ with $w_t\in\rinterior\dom h_t$ for all $t$. Firstly, if $I_{\bar h}^\theta (w)<\infty$, then there is $\bar w$ such that $\bar w_t=w_t$ outside a Borel $\mu^\theta$-null set and $\bar w_t\in\cl\dom h_t$ for all $t$. Secondly, let $\bar y\in C_b(\T;\rinterior\dom h)\cap\dom I_h$ so that $w^\nu=\frac{1}{\nu}\bar y+(1-\frac{1}{\nu})\bar w$ satisfies $w^\nu_t\in\rinterior\dom h_t$ for all $t$. It follows from convexity that $I_{\bar h}^\theta(w^\nu)\le \frac{1}{\nu}I^\theta_{\bar h}(\bar y)+(1-\frac{1}{\nu})I^\theta_{\bar h}(w)$, so $\lim_{\nu\rightarrow\infty} I^\theta_{\bar h}(w^\nu)\le I^\theta_{\bar h}(w)$. 

On the other hand, it follows from \eqref{eq:if3} that $I_{\bar h}^\theta(w)>-\infty$ for every $w\in L^\infty(\T,\mu^\theta;\R^d)$. Therefore it suffices to show that, for every $w\in L^\infty(\T,\mu^\theta;\R^d)$ with $I_{\bar h}^\theta(w)<\infty$ and $w_t\in\rinterior\dom h_t$ for all $t$ and for every $\epsilon>0$, there is $y\in C_b$ for which $I^\theta_{\bar h}(y)\le I^\theta_{\bar h}(w)+\epsilon$.

Since $\mu^\theta$ is Radon and since $I^\theta_{\bar h}(\bar y)$ and $I_{\bar h}^\theta(w)$ are finite, there is, by Lusin's theorem (see, e.g., \cite[Theorem 7.1.13]{bog7}), an open $\bar O\subset\T$ such that $\int_{\bar O} (|\bar h_t(\bar y_t)|+ |\bar h_t(w_t)|)d\mu^\theta<\epsilon/3$, $\bar O^C$ is compact and $w$ is continuous relative to $\bar O^C$. The mapping
\[
	\Gamma_t= \begin{cases}
		w_t\quad&\text{if } t\in \bar O^C\\
		\rinterior\dom h_t\quad&\text{if }t\in O
	\end{cases}
\]
is inner semicontinuous convex nonempty-valued so that, by \cite[Theorem 3.1''']{mic56}, there is a $y^w\in C$ with $y^w_t=w_t$ for all $t\in \bar O^C$ and $y^w_t\in\rinterior\dom h_t$ for all $t$. Since $\bar O^C$ is compact, there is an open $O^b\supset \bar O^C$ such that $y^w$ is bounded on $O^b$. 

Since, for all $t\in O^b$, there is an $O_t\in\mH_t$ with $\int_{O_t} |\bar h_t(y^w_t)|d\mu^\theta_t<\infty$, and since $\bar O^C$ is compact, there is an open set $O'\supset \bar O^C$ such that $B\mapsto \int_{O'\cap B} |\bar h_t(y^w_t)|d\mu^\theta_t$ is a finite measure, which is also Radon (see \cite[Lemma 7.1.11]{bog7}). Therefore \cite[Theorem 3.4]{gp80} implies the existence of an open $O\supset \bar O^C$ with $\int_{O\backslash \bar O^C} |\bar h_t(y^w_t)|d\mu_t^\theta<\epsilon/3$.

Define $\hat O=O\cap O^b$. Since $\hat O$ and $\bar O$ form an open cover of $\T$ and since $\T$ is normal, there is, by \cite[Theorem 36.1]{mun00}, a continuous partition of unity $(\hat \alpha,\bar\alpha)$ subordinate to $(\hat O,\bar O)$. We define $y=\hat \alpha y^w+\bar \alpha \bar y$ so that $y\in C_b$ and
\begin{align*}
	\int_\T \bar h_t(y_t)d\mu^\theta_t &\le \int_{\bar O^C} \bar h_t(w_t)d\mu^\theta_t+\int_{\hat O\backslash \bar O^C} \hat\alpha_t |\bar h_t(y^w_t)|d\mu^\theta_t+\int_{\bar O}\bar\alpha |\bar h_t(\bar y_t)|d\mu^\theta_t\\
	&\le\int \bar h_t(w_t)d\mu^\theta_t+\epsilon,
\end{align*}
which finishes the proof of the sufficiency.

Assume now that $\interior\dom h_t\neq\emptyset$ for all $t$, $I_h^*=J_{h^*}$ and that $J_{h^*}^*=I_h$. We prove outer $\mu$-regularity of $\dom h$ using Theorem~\ref{thm:ces}. Let $\tilde y\in C$ be such that $\tilde y_t\in\cl\dom h_t$ $\mu$-a.e. By Proposition~\ref{prop:mli} and Theorem~\ref{thm:mlim0}, $\tilde y\in C(\T;\mli\dom h)$ so that, for every $\nu\ge 1$, by \eqref{eq:inc}, the function $\tilde y^\nu=\frac{1}{\nu}\bar y+(1-\frac{1}{\nu})\tilde y$ satisfies $\tilde y^\nu\in C(\T,\interior\mli\dom h)$.

Fix $\nu$ and $t$. There is an $O_t\in\mH_t$ such that $\int_{O_t}|h_t(\tilde y_t^\nu)|d\mu_t<\infty$ and $\tilde y^\nu$ is bounded on $O_t$. Let $(\alpha^1,\alpha^2)$ be a continuous partition of unity subordinate to $(O_t,\T\backslash\{t\})$. We define $y^\nu=\alpha^1\tilde y^\nu+\alpha^2 \bar y$, which satisfies $y^\nu\in \dom I_h$. Therefore \eqref{eq:if2} implies that $y^\nu_{t'}\in\cl\dom h_{t'}$ for all ${t'}$. Since $y^\nu_t\rightarrow \tilde y_t$, we have that $\tilde y_t\in\cl\dom h_t$. Since $t$ was arbitrary, $\tilde y_t\in\cl\dom h_t$ for all $t$. By Theorem~\ref{thm:ces}, $\dom h$ is outer $\mu$-regular.
\end{proof}

When $\T\subset\R^n$, the following corollary generalizes \cite[Theorem 5]{roc71} by relaxing full lower semicontinuity of $\dom h$; see Lemma~\ref{lem:fisc}.
 \begin{corollary}\label{cor:ihcc}
 Assume that $\T\subset\R^n$ is compact, $\mu$ is a $\sigma$-finite Radon measure, $\dom h$ is inner semicontinuous with $\interior\dom h_t\neq\emptyset$ for all $t$ and that $I_h$ is finite on
\[
	\left\{y\in C\mid \exists r>0:\ \B(y_t,r)\subset \dom h_t\ \mu\text{-a.e.}\right\}.
\]
Then $\dom h$ is outer $\mu$-regular if and only if $I_h$ and $J_h^*$ are conjugates of each other. In this case $\interior \dom I_h=C(\T;\interior\dom h)$ (where the interior is with respect to the supremum norm).
 \end{corollary}
\begin{proof}
Let us verify the assumptions of Theorem~\ref{thm:if1}. By Proposition~\ref{prop:mli}, the mapping $t\mapsto \mli\dom h_t$ is inner semicontinuous solid convex-valued. Let $\hat y\in C(\T;\interior\mli\dom h)$; such $\hat y$ exists by \cite[Theorem 3.1''']{mic56}. By \cite[Lemma 2]{roc71}, there is an $r>0$ with $\hat y_t+x\in\interior\mli\dom h_t$ for all $t$ whenever $|x|<r$; see \cite[p.~460]{roc71}. By Proposition~\ref{prop:mli}, $\B(\hat y_t,r)\subset\dom h_t$ $\mu$-a.e. so that $\hat y\in\dom I_h$. Moreover, by \cite[Theorem 2]{roc71}, there is a $w\in L^1(\T,\mu;\R^d)$ with $I_{h^*}(w)<\infty$. By defining $\bar\theta\in M^b$ as $d\bar\theta/d\mu=w$ and $\bar\theta^s=0$, we have that $J_{h^*}(\bar\theta)<\infty$. Therefore, Theorem~\ref{thm:if1} is applicable, and consequently $\dom h$ is outer $\mu$-regular if and only if $I_h$ and $J_h^*$ are conjugates of each other.

Assume that $\dom h$ is outer $\mu$-regular. As above, we get that $\interior\dom I_h\supset C(\T;\interior\dom h)$. Assume that $\bar y\in\interior\dom I_h$. There is an $r>0$ such that $\bar y_t+x\in\dom h_t$ $\mu$-a.e. whenever $|x|<r$. By Theorem~\ref{thm:mlim0}, we have $\bar y_t+x\in\dom h_t$ for all $t$ whenever $|x|<r$. Thus $\bar y_t\in\interior\dom h_t$ for all $t$ and $\interior\dom I_h\subseteq C(\T;\interior\dom h)$.
\end{proof}

\subsection{Integral functionals on left continuous functions of bounded variation}\label{sec:left}
Let $\tau$ and $\tau_l$ be, respectively, the standard topology and the {\em left half-open topology} on $\R$. Here $\tau_l$ is generated by the basis $\{(s,t]\mid s<t\}$. The space $(\R,\tau_l)$ is a perfectly normal Lindel\"of $T_1$-space; in particular it is strongly Lindel\"of  and $C((\R,\tau_l);\R^d)$ is the space of left continuous functions (see \cite[p. 75]{ss95} \cite[p. 194]{eng89}).  We say that a set-valued mapping $\Gamma:[0,T]\rightrightarrows\R^d$ is {\em left-inner semicontinuous}, if it is inner semicontinuous with respect to $\tau_l$. Similarly we say that $\Gamma$ is {\em left-outer $\mu$-regular} if $\Gamma$ is outer $\mu$-regular with respect to $\tau_l$.

Let $(\Gamma^\alpha)_{\alpha\in\mathcal J}$ be a family of $\R^d$-valued set-valued mappings. A closed-valued mapping $\Gamma$ is called {\em $\mu\text{-}$essential supremum} of $(\Gamma^\alpha)_{\alpha\in\mathcal J}$ if $\Gamma^\alpha_t\subseteq\Gamma_t$ $\mu$-a.e. for all $\alpha\in\mathcal J$, and if $\tilde \Gamma$ is another closed-valued mapping satisfying $\Gamma^\alpha_t\subseteq\tilde\Gamma_t$ $\mu$-a.e. for all $\alpha\in\mathcal J$, then $\Gamma_t\subseteq \tilde\Gamma_t$ $\mu$-a.e. When $\mu$ is $\sigma$-finite, $\mu\text{-}$essential supremum exists (see, e.g., \cite[Theorem 1.3]{wan01}), it is defined almost everywhere and it is denoted by $\mu\text{-}\esssup_{\alpha\in\mathcal J}\Gamma^\alpha$.

We denote the space of left continuous $\R^d$-valued functions of bounded variation on $\R$ by $BV$.
\begin{theorem}\label{thm:BV}
Assume that $\mu$ is a $\sigma$-finite Radon measure on $([0,T],\tau)$, $h$ is a convex normal integrand on $[0,T]\times\R^d$ with $\dom h$ left-inner semicontinuous, $\interior\dom h_t\neq\emptyset$ for all $t$, and that $I_h$ is finite on the nonempty set
\[
	\left\{x\in BV\mid \exists r>0:\, \mathbb B(x_t,r)\subset\dom h_t\ \mu\text{-a.e. } t\in[0,T]\right\}.
\]
Then $\dom h$ is left-outer $\mu$-regular if and only if, for every $\theta\in M_b([0,T],\tau;\R^d)$,
\begin{align*}
	\sup_{x\in BV}\left\{\int_{[0,T]} x_td\theta_t-\int_{[0,T]} h_t(x_t)d\mu_t\right\} &=J_{h^*}(\theta).
\end{align*}
\end{theorem}
\begin{proof}
By \cite[Theorem 2]{roc71}, there is a $w\in L^1(\T,\mu;\R^d)$ with $I_{h^*}(w)<\infty$. Thus, as in \eqref{eq:if3}, $I_h(x)>-\infty$ for all $x\in BV$. Let $\bar B\in\mB(T)$ be such that $\mu(\bar B^C)=|\theta^s|(\bar B)=0$. We define
\begin{align*}
	\bar h_t(x)=
	\begin{cases} 
		-x\cdot\left(\frac{d\theta}{d\mu}\right)_t+h_t(x)\quad&\text{if } t\in \bar B\\
		-x\cdot \left(\frac{d\theta^s}{d|\theta^s|}\right)_t\quad&\text{if } t\in \bar B^C
	\end{cases}
\end{align*}
and $\mu^\theta=\mu+|\theta^s|$ so that 
\[
\int_{[0,T]} x_td\theta_t-\int_{[0,T]} h_t(x_t)d\mu_t=\int_{[0,T]} -\bar h_t(x_t)d\mu^\theta_t.
\]
The space $BV$ is PCU-stable in the sense that, for any $(x^i)_{i=0}^n\subset BV$ and for any continuous partition of unity $(\alpha^i)_{i=0}^n$ such that $\alpha^0\in C^\infty(\R;[0,1])$, $\alpha^i\in C_c^\infty(\R;[0,1])$, we have $\sum_{i=0}^n \alpha^i x^i\in BV$. By extending $\bar h$ to $\R\times\R^d$ by zero we get, by \cite[Theorem 1]{bv88}, that
\begin{align}\label{eq:bv0}
	\inf_{x\in BV} \int_{[0,T]} \bar h_t(x_t)d\mu^\theta_t &=\int_{[0,T]}\inf_{x\in\Gamma_t} \bar h_t(x)d\mu^\theta_t,
\end{align}
where $\Gamma=\mu^\theta\text{-}\esssup_{x\in BV\cap\dom I_h}x$.

Let us prove that $\Gamma_t=\mli\dom h_t$ $\mu^\theta$-a.e., where $\mli$ is with respect to $\tau_l$. By Proposition~\ref{prop:mli} and Theorem~\ref{thm:mlim0}, $x_t\in\mli\dom h_t$ for all $t$ whenever $x\in BV\cap\dom I_h$. Therefore it suffices to show that $\mli\dom h$ is smaller than the essential supremum. 

Let $\bar x\in BV$ and $\bar r>0$ be such that $\B(\bar x_t,\bar r)\in\dom h_t$ $\mu$-a.e. By Proposition~\ref{prop:mli}, $\mli\dom h$ is left-inner semicontinuous solid convex-valued. Let $t\in\R$ and $\hat x\in\interior\mli\dom h_t$. It follows from \cite[Theorem 5.9]{rw98} that there is an interval $(s,t]$ and $\hat r>0$ such that  $\B(\hat x,\hat r)\subset\interior\mli\dom h_{t'}$ for all $t'\in(s,t]$. By Proposition~\ref{prop:mli}, $\B(\hat x,\hat r)\subset\dom h_{t'}$ $\mu$-a.e. on $(s,t]$ so that there is an $r<\hat r$ such that $x_t=\bar x_t1_{[0,s]}+\hat x1_{(s,t]}+\bar x_t1_{(t,T]}$ satisfies $\B(x_t,r)\subseteq \dom h_t$ $\mu$-a.e. Therefore, for any $t$ and $\hat x\in\interior\mli\dom h_t$ there is $x\in BV\cap\dom I_h$ with $x_t=\hat x$. This implies that, for any $t$ and $x'\in\R^d$, 
\begin{equation}\label{eq:bv1}
d(x',\mli\dom h_t)=\inf\left\{|x'-x_t|\mid x\in BV\cap\dom I_h\right\}.
\end{equation}
Let $D$ be a countable dense set in $\R^d$. As in the proof of \cite[Theorem 1.3]{wan01}, there is a sequence $(x^\nu)_{\nu=1}^\infty\subset BV\cap\dom I_h$ such that $(\mu^\theta\text{-}\esssup_{x\in BV\cap\dom I_h} x)_t=\cl\bigcup_{\nu=1}^\infty x^\nu_t$ $\mu^\theta$-a.e. and, for every $x'\in D$,
\begin{align*}
	\inf\left\{|x'-x_t|\mid x\in BV\cap\dom I_h\right\} \geq \inf\left\{|x'-x^\nu_t|\mid \nu\geq 1\right\}\quad \mu^\theta\text{-a.e.}
\end{align*}
This together with \eqref{eq:bv1} implies that $\mli\dom h_t\subseteq ({\mu^\theta\text{-}\esssup}_{x\in BV\cap\dom I_h} x)_t$ $\mu^\theta$-a.e.

Since $\mli\dom h_t=\Gamma_t$ $\mu^\theta$-a.e., and $I_{\bar h}^\theta$ is proper on $BV$, we get, by \eqref{eq:bv0}, that
\begin{align*}
\sup_{x\in BV}\left\{\int_\R x_td\theta_t-\int_{\R} h_t(x_t)d\mu_t\right\} &=-\inf_{x\in BV} \int_\T \bar h_t(x_t)d\mu^\theta_t\\
 &=-\int_\T\inf_{x\in\mli\dom h_t} \bar h_t(x)d\mu^\theta_t.
\end{align*}

Assuming that $\dom h$ is left-outer $\mu$-regular, similarly to \eqref{eq:barh}, the above expression implies the claim. Assume now that $\dom h$ is not left-outer $\mu$-regular. By \eqref{eq:inc} and the fact that $(h_t^*)^\infty(v)=\sup\{v\cdot x\mid x\in \dom h_t\}$ (see \cite[Theorem 13.3]{roc70a}), there is $\bar t$ and $\bar v\in\R^d$ such that $\sup\{\bar v\cdot x\mid x\in \mli\dom h_{\bar t}\}>(h_{\bar t}^*)^\infty(\bar v)$. Here $\mu$ cannot have an atom at $\bar t$, otherwise $\{\bar t\}\in\mH^{\mu\#}_{\bar t}$ and $\mli\dom h_{\bar t}\subseteq\cl\dom h_{\bar t}$. Let $\theta=\bar \theta+\bar v\delta_{\bar t}$, where $\delta_{\bar t}$ denotes a Dirac measure at $\bar t$ and $\bar\theta$ is absolutely continuous with respect to $\mu$ with $d\bar\theta/d\mu=w$. By Proposition~\ref{prop:mli}, $\cl\dom h_t=\mli\dom h_t$ $\mu$-a.e. so that
\begin{align*}
	\sup_{x\in BV}\left\{\int_\R x_td\theta_t-\int_{\R} h_t(x_t)d\mu_t\right\} &=-\int_\T\inf_{x\in\mli\dom h_t} \bar h_t(y)d\mu^\theta_t\\
	&=\int_\T h_t^*\left(\left(\frac{d\theta}{d\mu}\right)_t\right)d\mu_t+\sup_{x\in\mli\dom h_{\bar t}} \bar v\cdot x\\
	&> J_{h^*}(\theta).
\end{align*}
\end{proof}

\bibliographystyle{spmpsci} 
\bibliography{sp}
\end{document}